\documentclass{article}

\usepackage{fullpage}
\usepackage{floatrow}

\usepackage{amssymb}
\usepackage{amsthm}
\usepackage{amsmath}

\usepackage{graphicx}

\newcommand{\Z}{\mathbb{Z}}
\newcommand{\R}{\mathbb{R}}
\newcommand{\N}{\mathbb{N}}
\newcommand{\Equ}{\operatorname{Equ}}

\newtheorem{theorem}{Theorem}[section]
\newtheorem{lemma}[theorem]{Lemma}

\theoremstyle{definition}

\newtheorem{cor}[theorem]{Corollary}

\theoremstyle{remark}

\numberwithin{equation}{section}

\begin{document}

\title{An ergodic algorithm for generating knots with a prescribed injectivity radius}


\author{Kyle Leland Chapman}




\begin{abstract}
The first algorithm for sampling the space of thick equilateral knots, as a function of thickness, will be described. This algorithm is based on previous algorithms of applying random reflections. It also is an off lattice equivalent of the pivot algorithm. To prove the efficacy of the algorithm, we describe a method for turning any knot into the regular planar polygon using only thickness non-decreasing moves. This approach ensures that the algorithm has a positive probability of connecting any two knots with the required thickness constraint and so is ergodic. This ergodic sampling unlocks the ability to analyze the effects of thickness on properties of the geometric knot such as radius of gyration and probability of unknotting.
\end{abstract}

\maketitle

\section{Introduction}

\subsection{Discussion of the applications of Equilateral Knots}

Knot theory is a discipline which has seen a large number of applications, particularly in recent years. The study of knots and links is useful for disciplines such as biology, polymer physics, and materials enginering. Many of the physical properties of structures in these fields comes from the knotting and linking of molecules rather than the actual chemical structure\cite{McLeish08}. In biology, it was first discovered that the bacteriophage \(\varphi\)X174 has DNA with a closed loop structure\cite{FiersSinsheimer62} similar to that pictured in figure~\ref{fig:DNAPicture} and since then numerous other examples have been found. Further, the large structure of DNA means that even if it is technically possible for two open chain DNA strands to unlink, it is not feasible. This led to the discovery by James C. Wang of enzymes called topoisomerases which alter the knotting of DNA allowing for cellular division and indicating that the knotting of DNA impacts function \cite{Champoux01}\cite{LDW76}. Most attempts at proving facts about applications of knot theory have been restricted to topological knots, where we have eliminated physical factors such as friction, rigidity, torsion, thickness, and arclength. Some work has been done to look at arclength and local homogeneity by focusing on equilateral knots of a fixed number of edges \cite{Millett94}\cite{Randell88b}. Ideally, however, we would like to study knots with more of these properties included. Any information could act as a foundation upon which to make concrete statements about knots with these physical properties. Unfortunately, even randomly sampling the space of knots with these properties poses an incredible challenge. This paper will provide a step towards a better sampling of knots with geometric properties, specifically thickness, homogeneity, and arclength.

\begin{figure}
\includegraphics{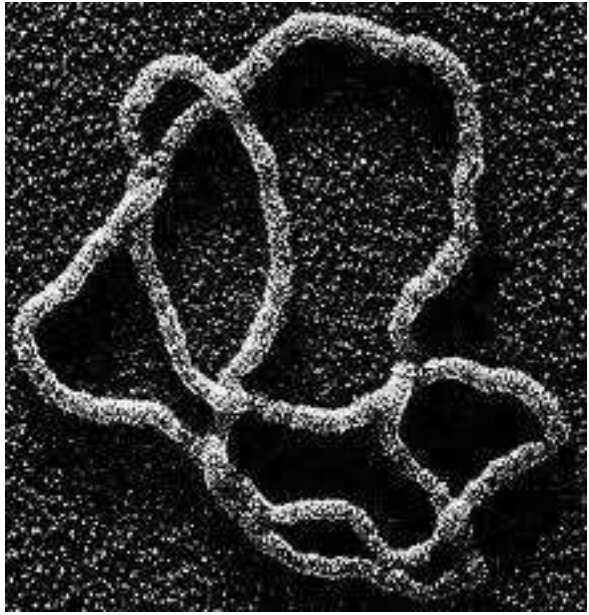}
\caption{An electron micrograph of a six crossing knot in DNA\cite{WDC85}}
\label{fig:DNAPicture}
\end{figure}

\subsection{Previous Thickness Free Generation Methods}

There are a number of methods available for the random generation of equilateral knots \cite{ACM11}. Because we seek to have the arc close to give a ring polygon, we can't merely generate an arbitrary random walk. Most methods come in two stages, an initialization stage, and a randomization stage. Two similar methods for initializing a closed polygon are the generalized hedgehog method, and the triangle method\cite{ACM11}. These involve randomly generating sets of two or three vectors which add to zero, and then randomly ordering these vectors. Alternatively one can start just with the regular planar polygon. Any of these methods give a closed loop for a starting position, but fail to generate an arbitrary polygonal knot. From these initial positions, we can randomize by applying a sequence of moves to a starting polygon, such as polygonal folds, reflections, or crankshafts\cite{ACM11}. All of these moves preserve the fact that the knot closes, as well as the edge lengths and homogeneity. Unfortunately, none of these methods of direct generation or movement through knot space have been rigorously shown to give or converge to the appropriate measure on knot space. There is a method using symplectic geometry which has been shown by Cantarella and Shonkwiler to converge to the correct measure\cite{CS13}. All of these generation methods allow us to sample closed knots with a fixed edge length. These methods do not, however, allow us to have any control over the thickness of the knots generated.

\subsection{Thickness and Excluded Volume Problem}

A knot is traditionally viewed as being an infinitesimally thin strand or loop, represented by mappings of a circle into space. In applications, however, we need knots which have a thickness, which means that they are maps of cylinders into space. This means that portions of the knot take up volume, excluding the rest of knot from occupying that space. The natural naive attempt at sampling knots with a thickness constraint is to merely generate knots using another method, such as the symplectic method, and then restrict to those samples satisfying the thickness constraint. This method is extremely ineffective in practice. This is because the probability of a knot that is randomly generated satisfying interesting thickness constraints can be very low. For example, in a sample of 10,000 random decagons, 6,990 had a thickness less than .01, compared to a maximal thickness of .1539. This problem only gets worse as the number of edges increases. Even in the case of open chains with thickness, where we can generate a chain and use that the subchain already generated will have to have satisfied the thickness constraint, we run into difficulties as it is possible for such generation methods to become trapped. This is a direct result of the problem of excluded volumes, as regions in space may become too filled with portions of the knot which were previously generated to allow for passage by a tube of the chosen thickness constraint. This returns us to a problem faced by closed chains, as you must make numerous attempts to generate each sample. Each of these methods can become extremely time consuming when trying to generate adequate sample sizes.

\subsection{Lattice Methods}

One resolution to the issue of excluded volume has been to restrict to polygons on an integer lattice. This gives efficient sampling, but insufficient work has been done on justifying that the results of sampling lattice knots are representative of the results for off-lattice knots with excluded volume. For those who are familiar with the methods of lattice polygons, many pieces of this paper are reminiscent of the proof that the pivot algorithm is ergodic, however, with a few core differences. These essentially all arise from the fact that lattice polygons of a fixed edge length form a finite state space, while thick polygons off of a lattice form a manifold. This means that it is important to consider the topology of the space and the topology of the moves.

\subsection{Resolution of this Obstacle and Outline of this Paper}

In this paper, we give a method of generating knots with a thickness constraint. It allows us to check the thickness throughout the generation procedure. It is built on the standard reflection algorithm, which starts with the regular planar polygon, and applies a sequence of reflection moves to generate a random knot. We check the thickness after a specific number of reflections to make sure the thickness constraint we desire is still maintained. Verifying that this will randomly sample the whole space of knots of a given thickness is the focus of this paper. While the act of checking thickness increases computation time, this method is still an improvement over the extremely low yield of previous algorithms.

In section 2, we outline all the important definitions. In section 3, we define and prove the algorithm for bringing any knot to the regular planar polygon using thickness non-decreasing moves. In section 4, we use the algorithm from section 3 to prove that a Monte Carlo Markov Chain built from reflections is ergodic for any strictly positive thickness constraint.

\section{Definitions}
\subsection{Knots}

In this paper we use the following definitions. We will be considering knots as piecewise linear equilateral polygons in 3-space, such as the unknot pictured in figure~\ref{fig:SampleKnot}, rather than the more common notion of smooth or topological knots, since the class of piecewise linear equilateral knots is more useful for simulation and applications to the study of macromolecules, and is much more tractable for computer based analysis. A knot with \(n\) edges will be defined to be a sequence of \(n\) points, called its vertexes, \(\{v_i\}_{i\in \Z_{n}}\) in \(\R^3\) with the property that \(\|v_{i+1}-v_{i}\| = 1\) for each index, and for which the line segments connecting pairs of adjacent vertexes intersect in, at most, a common vertex. The indexing is in \(\Z_n\) so that the result is a closed loop. The space of such knots will be denoted \(\Equ(n)\), and be given the subset topology of \(\R^{3n}\). Knots in \(\Equ(n)\) will be considered equivalent if they differ by an affine transformation, which is a composition of translations and rotations. These transformations are the orientation preserving component of the isometry group for \(\R^3\) and so preserve all lengths and angles of the knot. This equivalence is to account for the fact that the knot properties should be independent of where the origin is or the choice of oriented orthogonal basis. We will also, during sampling, consider equivalence up to the full isometry group, since this results in only a double covering of the space we seek to sample. An \emph{arc} of the knot is a subsequence of the knot given by \(\{v_i\}_{i=j}^{k}\). Due to the cyclic indexing, this corresponds to the traditional notion of a subarc of a knot. Given an arc \(a = \{v_i\}_{i=j}^{k}\) we will define the complementary arc to be \(a^{c} = \{v_i\}_{i=k}^{j+n}\).

\begin{figure}
\includegraphics{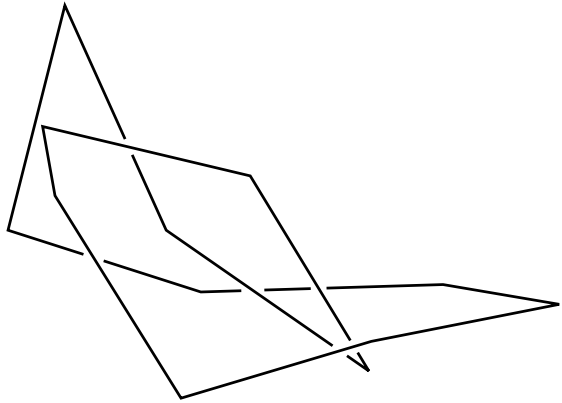}
\caption{A sample random 12-gon}
\label{fig:SampleKnot}
\end{figure}

The \emph{interior angle} at a vertex, \(v_i\) is defined to be the angle between the two vectors \(v_{i+1} - v_i\) and \(v_{i-1} - v_i\). This angle will be called \emph{regular} if it equals \(\pi(n-2)/n\), large if it is larger than regular and small if it is smaller than regular. The \emph{turning angle} of a vertex is \(\pi\) minus the interior angle of the vertex. This gives us a means of discussing the curvature of a piecewise linear knot, as being the turning angle at a vertex. These concepts are demonstrated in figure~\ref{fig:TurningAngle}.

\begin{figure}
\includegraphics{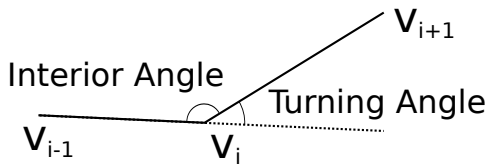}
\caption{The two angles of importance at each vertex}
\label{fig:TurningAngle}
\end{figure}

\subsection{Thickness and Curvature}

The polygonal injectivity radius will be as defined by Rawdon in \cite{Rawdon2000}, which is chosen in such a way as to limit upon the injectivity radius of a smooth curve. Recall that for a smooth curve, the injectivity radius is the largest \(r\) such that a disk of radius \(r\) perpendicular to the curve can be centered at every point of the knot simultaneously without intersections between distinct disks. The two limiting factors to this are the curvature, which will cause close disks to intersect, and long range interactions. For this reason, there are two types of radii we must consider for polygonal knots, taking their minimum. For an equilateral knot \(K\), the short range radius, \(MinRad(K)\), is defined as the minimum over all vertexes of \(d_v = \frac{1}{2} \tan(\theta_v /2)\) as in figure~\ref{fig:ShortDistances}, where \(\theta_v\) is the interior angle at the vertex \(v\). This is the maximum radius of disks which can be placed perpendicular to the midpoints of the two adjacent edges without intersection. This also means that an arc of a circle with radius \(d_v\) can be inscribed in the pair of edges adjacent to \(v\) and meeting it at the midpoints of the edge. For any radius \(r < d_v\), the same can be said, except that the intersection with the edges will be closer to the vertex.

\begin{figure}
\includegraphics{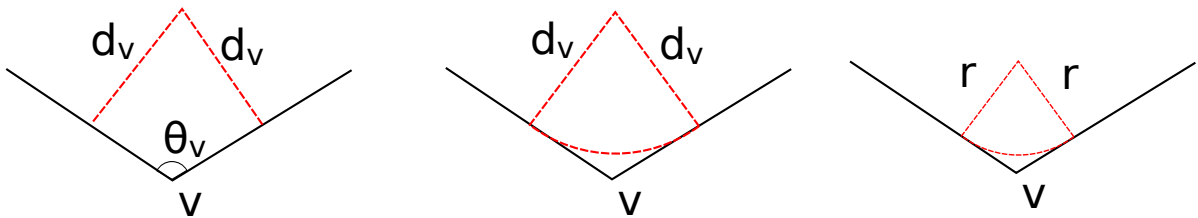}
\caption{The left diagram shows the perpendicular bisectors which intersect at a distance of \(d_v\). The middle and right diagram show inscribing an arc of a circle with radius \(d_v\) and \(r < d_v\).}
\label{fig:ShortDistances}
\end{figure}

A pair of points on the knot, \(a,b\) will be called a \emph{doubly critical pair} if \(a\) is a local extrema of the distance function to \(b\) and \(b\) is a local extrema of the distance function to \(a\), as in figure~\ref{fig:LongDistances}. The set of all such pairs is \(DC(K)\). The long range radius is half of the doubly critical self distance \(dcsd(K) = min_{(a,b)\in DC(K)}(\|a-b\|)\). The \emph{injectivity radius} is defined by \(R(K) = min(MinRad(K),dcsd(K)/2)\). As in the smooth case, \emph{thickness} is the injectivity radius divided by the arclength. We will denote the space of equilateral knots with thickness greater than or equal to \(t\) by \(Equ(n,t)\).

\begin{figure}
\includegraphics{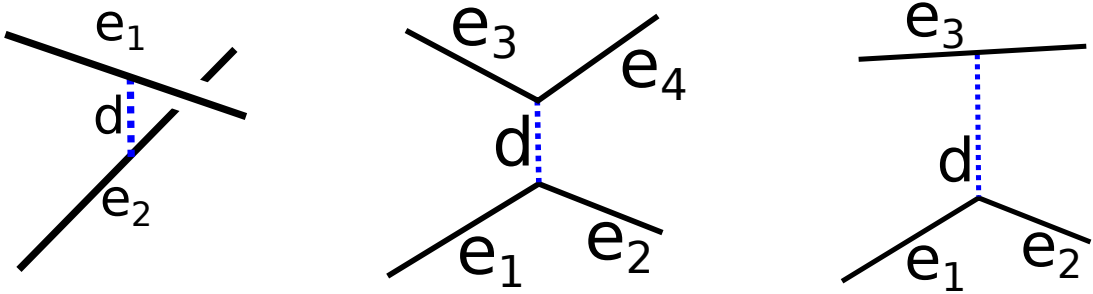}
\caption{The three cases for doubly critical self distance, a pair of skew edges, a pair of vertexes, and a vertex-edge pair.}
\label{fig:LongDistances}
\end{figure}

The \emph{total curvature} between two points of the realization \(tc(a,b)\) is the minimum over the two arcs connecting \(a\) to \(b\) of the sum of turning angles for vertexes between \(a\) and \(b\), including the turning angles at \(a\) and/or \(b\) if either is a vertex. We will use this notion of curvature by applying Schur's theorem for sectionally smooth curves.

\begin{theorem}[Schur's Theorem for Sectionally Smooth Curves \cite{Chern70}]
\label{thm:Schur}
Suppose \(C\) and \(C^*\) are two sectionally smooth curves of the same length parametrized by arclength, with \(\kappa,\kappa^*\) being their curvatures within the smooth sections, \(\alpha,\alpha^*\) being their turning angles at the vertexes between smooth section and \(C\) together with its chord forming a simple convex planar curve. If \(\kappa^*(s) \leq \kappa(s)\) and \(\alpha^*(v) \leq \alpha(v)\), then the end to end distance of \(C\) is less than or equal to the end to end distance of \(C^*\).
\end{theorem}

One consequence of Schur's theorem is that if we have an arbitrary curve \(C^*\), and we do not decrease the pointwise curvatures and turning angles, that moving to a planar convex curve \(C\) cannot increase end-to-end distance. This condition of having a convex planar curve to compare to is certainly satisfied if the total curvature between points is less than or equal to \(\pi\). The set of points which are separated by at least \(\pi\) in turning angle will be denoted \(TC(K) = \{(a,b)|tc(a,b) > \pi\} \subseteq K\times K\). This definition, together with Schur's theorem, is sufficient to prove a lemma from \cite{Rawdon2000} which we will use in a couple of places.

\begin{lemma}
\label{lem:boundaryTurning}
If \(a,b\) are a pair of vertexes of \(K\) on the boundary of \(TC(K)\), then \(\|a-b\| \geq 2 MinRad(K)\).
\end{lemma}

This lemma is listed in \cite{Rawdon2000} as lemma 22 and has a proof which relies mainly on using Schur's theorem to compare the arc from \(a\) to \(b\) in \(K\) to a planar arc with the same turning angles

This lets us prove another lemma, from \cite{Rawdon2000}, that will make it much easier to determine when thickness has not been decreased. It allows us to expand our consideration from pairs of points which are doubly critical, to those which are separated by enough curvature. This is a much easier collection of points to directly compare.

\begin{lemma}
\label{lem:TCnotDCSD}
For a polygonal knot \(K\) with injectivity radius \(R(K)\), \[R(K) = min( MinRad(K) , min_{(a,b)\in TC(K)}(\|a-b\|)/2)\]
\end{lemma}

This is listed in \cite{Rawdon2000} as Theorem 10.

\begin{cor}
\label{cor:PlanarThickness}
The injectivity radius of a convex planar polygon \(K\) is \(MinRad(K)\)
\end{cor}

\begin{proof}
In a convex planar polygon \(K\), \(TC(K)\) is empty and so the minimum is determined only by the first term, \(MinRad(K)\).
\end{proof}

This definition of thickness has three nice properties which indicate that it is an appropriate choice for our definition of polygonal thickness.
\begin{itemize}
	\item If \(r < R(K)\), then the radius r neighborhood of K is a torus which deformation retracts onto K\\
	This is theorem 7.1 in \cite{MPR2007}
	\item If a sequence of polygonal knots \(P_n\) converge to a smooth knot \(K\), then \(\lim_{n\rightarrow\infty} R(P_n) = R(K)\)\\
	This is theorem 6 in \cite{Rawdon2000}
	\item R is a continuous function, so in particular \(\Equ(n,t)\) is compact\\
	This is theorem 11 in \cite{Rawdon2000}
\end{itemize}

\subsection{Reflection Moves}

An arc reflection or, simply, a \emph{reflection move} consists of making a choice of plane containing two vertexes \(x_i, x_j\) and reflecting the vertexes \(\{x_k\}_{k=i}^j\) across that plane to get a new knot. For many of the planes we will choose in the next section, the full arc \(\{x_k\}_{k=i}^j\) lies on one side of the plane of reflection, but this is not required for the move to be defined. If parts of the arc lie on both sides of the plane of reflection, this simply means the vertexes in the arc switch sides of the plane of reflection. We also note that reflecting an arc across a plane, and reflecting the complementary arc across the same plane differ by a reflection of the whole knot across that plane, which is a global isometry. Thus, we may choose with each reflection move to reflect either arc connecting the chosen vertexes. 

The result of a reflection move may not in general be a non-singular knot, as self-intersections may be created, but we will still consider such non-embedded curves as knots with zero thickness. We also will talk about neighborhoods of a reflection move, where the space of reflection moves on a knot is given by a disjoint union of circles. This is because each pair of vertexes defines a line which the reflection plane must contain, and the planes through a given fixed line form a circle.

\section{The Algorithm}
We can consider applying random reflection moves to a knot \(K \in \Equ(n,t)\). Sometimes these random reflection moves will result in a new knot in \(\Equ(n,t)\) and sometimes it will not. We will show that if we apply random reflections, then every neighborhood of the regular planar polygon has a positive probability of being reached in finite time from every starting position, without leaving \(\Equ(n,t)\) for more than \(6\) reflections. For this, we will show that for any knot \(K\) in \(Equ(n,t)\), there is a sequence of moves \(m_i\) consisting of up to \(6\) reflection moves each, giving a sequence of knots \(\{K_i\}_0^N\) also in \(\Equ(n,t)\) with \(K=K_0\), \(K_N\) the regular planar polygon, each \(K_{i+1} = m_i(K_i)\) and each \(m_i\) having an open neighborhood \(M_i\) with \(m(K_i) \in Equ(n,t)\) for every \(m\in M_i\).

First, we will show that there is a way of spreading out the knot so that it has a projection which is convex. We then start flattening the knot until it is planar. Finally, we will do moves to make each angle regular, resulting in the regular planar polygon. An outline of the process is shown in figure~\ref{fig:ThreeSteps}. During the first two of these steps, making the knot into a planar convex polygon, we will be changing the knot in a way which makes every pair of points further apart. In these cases we use the following lemma to show we have not decreased thickness.

\begin{lemma}
\label{lem:distance}
If \(K,K' \in \Equ(n)\) and there is an arclength preserving function \(f:K\rightarrow K'\) such that \(d(f(a),f(b)) \geq d(a,b)\) for every pair of points \(a,b \in K\), then the thickness of \(K'\) is greater than or equal to the thickness of \(K\)
\end{lemma}

\begin{proof}
First, consider the interior angles. We will examine a pair of points, \(a,b\), one immediately before and one immediately after a vertex, \(v\). The distance between these two points cannot decrease, while their distance to the vertex is preserved. These quantities determine the interior angle of that vertex and show that it cannot decrease. Thus, the interior angles are all non-decreasing.
	
Second, we use Lemma~\ref{lem:TCnotDCSD} which says that rather than only checking long range thickness at points which locally minimize distance, we can look at the long range thickness between any pair of points separated by a total of \(\pi\) turning angles. We have already established that each interior angle is non-decreasing, which means that the turning angles are non-increasing. This tells us that if \(tc(a,b)\leq \pi\) then \(tc(f(a),f(b))\leq \pi\), and so \(TC(K') \subseteq f(TC(K))\). Further, since every pair of points is made no closer, we have the following
	
\begin{align}
	& min_{(a,b)\in TC(K)}(\|a-b\|)\\
	\leq \:& min_{(f(a),f(b)) \in TC(K')}(\|a-b\|) \\
	\leq \:& min_{(f(a),f(b)) \in TC(K')}(\|f(a)-f(b)\|) \\
	= \:& min_{(a,b)\in TC(K')}(\|a-b\|)
\end{align}
Therefore the long range thickness is non-decreasing. Thus, the result is shown.
\end{proof}

\begin{figure}
\includegraphics{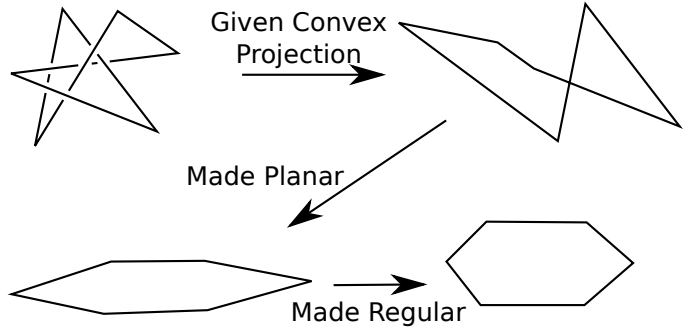}
\caption{An example of three steps applied to a random hexagon}
\label{fig:ThreeSteps}
\end{figure}

\subsection{Convex Projections}

Throughout this portion of the paper, we would like to know that expanding the knot does not decrease the thickness. In practice, this amounts to showing the injectivity radius is not decreased, since throughout the process the arclength is preserved. We will use the following lemma multiple times.

\begin{lemma}
\label{lem:Thickness}
If \(r\) is a reflection move on a knot \(K\) across a plane \(P\) with \(P\) not intersecting the interior of the convex hull of \(K\), then the thickness of \(r(K)\) is no less than the thickness of \(K\).
\end{lemma}

\begin{proof}
By Lemma~\ref{lem:distance}, it suffices to show that the distance between any two points is not decreased by \(r\). We can also conjugate by an affine transformation, which means that without loss of generality we can choose \(P\) to be the \(x-y\) plane, and \(K\) living in the upper half space. If both points are fixed or both points move, then the distance between them is unchanged, so it suffices to consider pairs of points where exactly one of the two is moved. We denote these two points \((x,y,z)\) and \((x',y',z')\) with \(z,z' \geq 0\). Thus, the lemma has been reduced to showing that \(d((x,y,z),(x',y',z')) \leq d((x,y,z),(x',y',-z'))\).
\begin{align}
z, z' &\geq 0 \\
2zz' &\geq -2zz'\\
z^2+2zz'+z'^2 &\geq z^2 - 2zz'+z'^2\\
(z+z')^2 &\geq (z-z')^2\\
(x-x')^2+(y-y')^2+(z-(-z'))^2 &\geq (x-x')^2+(y-y')^2+(z-z')^2\\
d((x,y,z),(x',y',-z'))^2 &\geq d((x,y,z),(x',y',z'))^2
\end{align}
Thus, the distance between pairs of points is not decreased by \(r\) and so the thickness of \(r(K)\) is no less than the thickness of \(K\).
\end{proof}

We now proceed with applying a sequence of reflection moves to a knot \(K\) so that some projection \(p(K)\) is convex. We choose the orthogonal projection into the \(x-y\) plane, giving us \(p(K) \subseteq \R^2\). We want to choose a fixed orthogonal projection for two reasons. First, it is useful that a reflection in the plane will correspond to a reflection in space, which requires that the projection is orthogonal. Second, it is important that we have a well defined and consistent set of minimum height vertexes. This means that we may apply a rotation to the whole knot before beginning, but once we have made a choice of projection, we must stick to it. This, unfortunately, will prevent us from using arguments about general position, but such methods can be safely avoided. Because we have done nothing special to the knot beforehand, this map may be terribly non-injective and perhaps is not even an immersion. Regardless, it is a piecewise linear map into the plane, about which many results are known. Our preliminary goal will be to apply reflection moves to \(K\) so as to make \(p(K)\) convex, while maintaining certain properties along the way. 

We need a few definitions for this section specifically. 
\begin{itemize}
	\item A vertex of a polygon in \(\R^2\) is \emph{exposed} if it is also a vertex of the convex hull.
	\item A pair of vertexes is an \emph{exposed pair} if they are both exposed and share a line which does not intersect the interior of the convex hull and which does not contain either arc connecting them.
	\item A map \(f:S^1\rightarrow \R^2\) is \emph{convex} if it is an embedding onto the boundary of the convex hull of its image.
	\item A map \(f:S^1\rightarrow \R^2\) a map is \emph{nearly convex} if its image is contained in the boundary of its convex hull
	\item A nearly convex map is \emph{exposed} if the pre-image of each vertex of the convex hull of the image is connected.
	\item A map \(f:S^1\rightarrow \R^2\) is \emph{subdimensional} if the convex hull of its image has empty interior.
\end{itemize} 

We can be visualize these using figure~\ref{fig:ConvexityTypes}. A nearly convex map might run back and forth through a proper subset of the boundary of the convex hull or zig zag back and forth within an edge before continuing. An exposed polygon, on the other hand, will have eliminated the cases such as moving back and forth through a proper subset or winding about the convex hull multiple times, but we may still have zig-zagging within an edge of the boundary of the convex hull. We will quote a theorem of two-dimensional geometry before proving a three dimensional analog.

\begin{figure}
\includegraphics{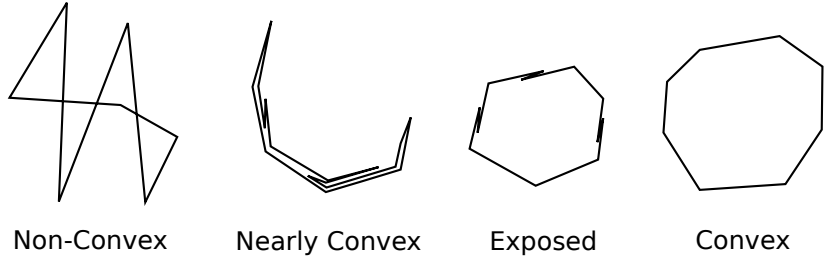}
\caption{Examples of the types of convexity defined in this section. Overlapping edges are separated for clarity.}
\label{fig:ConvexityTypes}
\end{figure}

\subsubsection{Gr\"{u}nbaum-Zaks Theorem}

We will be basing our argument about three dimensional knots on the following theorem.

\begin{theorem}
\label{thm:GrZaks}
Every polygon in the plane, not necessarily embedded, can be transformed into an exposed polygon by a finite sequence of reflections, determined at each step by an exposed pair of vertices
\end{theorem}

\begin{figure}
\includegraphics{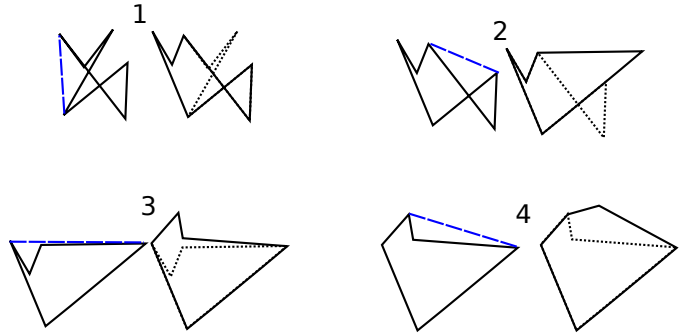}
\caption{An example of a polygonal loop which can be convexified in 4 reflections. Each reflection move is shown using two figures. The left is the polygon before the reflection with a supporting line connecting exposed vertexes highlighted in  blue. The right polygon is the result of the move and shows the previous location of the moved arc as a dotted line. Note that the initial configuration includes a partial doubling back, a type of non-generic intersection which can be accounted for by the theorem.}
\label{fig:Convexify}
\end{figure}

This theorem was proven in 2001 by Gr\"{u}nbaum and Zaks\cite{Grunbaum2001333}, appearing as Theorem 3. We will use the ability to convexify polygonal loops in the plane in Theorem~\ref{thm:Convexify} to make the orthogonal projection of the knot into the \(x-y\) plane exposed. We require the argument of this theorem which is stronger than earlier results such as the Erd\H{o}s-Nagy Theorem\cite{Toussaint2005219}, which applies to embedded polygons, because we need to be able to consider intersections in the projection, such as multiple vertexes going to the same point, or edges intersecting in segments as in Figure~\ref{fig:Convexify}. Our needs also require different choices of vertexes.

We prove a sequence of lemmas that, when taken together, prove the desired theorem about knots in space. We first show that if the projection of our knot isn't exposed, it has an exposed pair.

\begin{lemma}
\label{lem:ExposedPair}
If \(p(K)\) is a polygon in the plane which is not exposed, then it has an exposed pair.
\end{lemma}

\begin{figure}
\includegraphics[]{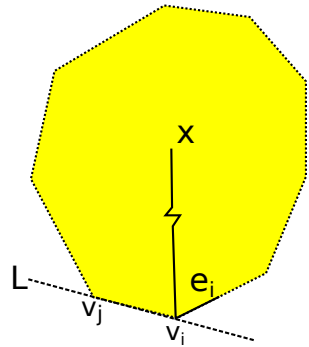}
\caption{An outline of the process for finding an exposed pair on a non-convex polygon}
\label{fig:ExposedPair}
\end{figure}

\begin{proof}
First, consider the case where \(p(K)\) is not nearly convex. This means that there is some point \(x\) of \(p(K)\) on the interior of \(conv(p(K))\), as in figure~\ref{fig:ExposedPair}. This also means \(conv(p(K))\) is not subdimensional, which means that \(p(K)\) has at least three non-collinear exposed vertexes. This point \(x\) need not be a vertex of \(p(K)\). From \(x\), proceed forward to the first exposed vertex \(v_i\). The vertex \(v_i\) has two supporting lines of \(conv(p(K))\) coming out of it. The edge \(e_i\) coming out of \(v_i\) cannot lie in both, so choose a supporting line \(l\) of \(conv(p(K))\) which contains \(v_i\) but does not contain \(e_i\). This line must hit a second exposed vertex \(v_j\). The line \(l\) then does not contain either arc connecting \(v_i\) to \(v_j\), as it does not contain \(e_i\) and it does not contain \(x\), and by the choice of \(v_i\) as the first exposed vertex after \(x\), these must lie in different arcs connecting \(v_i\) to \(v_j\). Therefore, \(v_i\) and \(v_j\) form an exposed pair.

Next, consider the case where \(p(K)\) is nearly convex but not exposed. This means that there is some vertex of \(conv(p(K))\) with disconnected preimage. This tells us that there are \(v_i,v_j\) with \(p(v_i) = p(v_j)\) but with neither arc connecting \(v_i\) to \(v_j\) mapped to \(p(v_i)\). Let \(l\) be a line which intersects \(conv(p(K))\) in only this exposed vertex. Then \(l\) is a line which does not intersect the interior of \(conv(p(K))\) and which does not contain either arc connecting \(v_i\) to \(v_j\) and so \(v_i,v_j\) form an exposed pair.
\end{proof}

Define a plane \(P\) to be a \emph{boundary plane} of \(K\) if it is vertical and does not intersect the interior of \(conv(K)\), as in figure~\ref{fig:SampleReflection}. This means it is the projection preimage of a supporting line of \(conv(p(K))\). We also define \(v_i,v_j\) in \(K\) to be an \emph{edge pair} if they share an edge of \(conv(K)\) and a boundary plane \(P\) which does not contain either arc connecting \(v_i\) to \(v_j\). We next show that if the projection of our knot isn't exposed, then there is an edge pair.

\begin{lemma}
\label{lem:CommonEdge}
If \(K\) is a knot with \(p(K)\) not exposed, then there is an edge pair \(v_i,v_j\).
\end{lemma}

\begin{figure}
\includegraphics[]{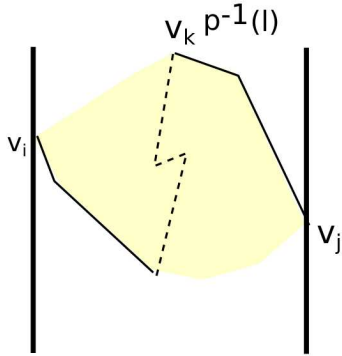}
\caption{An example of a vertical face of the convex hull of the knot, with an exposed pair \(v_i,v_j\), as well as the intermediate vertex which forms the edge pair \(v_i,v_k\).}
\label{fig:EdgePair}
\end{figure}

\begin{proof}
Since \(p(K)\) is not exposed, Lemma~\ref{lem:ExposedPair} tells us there is an exposed pair of vertexes \(p(v_i),p(v_j)\). This exposed pair shares a line \(l\) with \(l\) disjoint from the interior of \(conv(p(K))\), and not containing either arc connecting \(p(v_i)\) to \(p(v_j)\). Looking at projection preimages, as in figure~\ref{fig:EdgePair}, this means that there is a boundary plane \(P = p^{-1}(l)\) with \(P\) containing \(v_i\) and \(v_j\), but with \(P\) not containing either arc connecting them. As their projections are exposed, they must share a vertical face \(F\) of \(conv(K)\), and both lie on the boundary of that face.

Choose one arc around the boundary of \(F\). This will contain some number of intermediate vertexes. If there are no such vertexes, then \(v_i\) and \(v_j\) share an edge of \(conv(K)\) and we are done. Otherwise, choose an intermediate vertex \(v_k\). Denote the arcs connecting \(v_i\) to \(v_j\) by \(A,B\). The vertex \(v_k\) lies on one of those two arcs, and so without loss of generality, suppose it is \(B\). Thus, \(B = B_1\cup B_2\), where \(B_1\) is an arc connecting \(v_i\) to \(v_k\) and \(B_2\) is an arc connecting \(v_k\) to \(v_j\). The arc \(B\) does not lie in \(P\) so either \(B_1\) does not lie in \(P\) or \(B_2\) does not lie in \(P\). If \(B_1\) does not, then \(v_i\) and \(v_k\) are a pair of vertexes in \(P\) which both lie on a boundary arc of \(F\) but with fewer intermediate boundary vertexes. If \(B_2\) is not contained in \(P\), then \(v_k\) and \(v_j\) are a pair of vertexes in \(P\) which both lie on a boundary arc of \(F\) but with fewer intermediate boundary vertexes. Thus, in every case where there is an intermediate vertex on the boundary of \(F\), we can choose new vertexes \(v_i\) and \(v_j\) in \(P\) with neither arc connecting them in \(P\) but with fewer vertexes between them. This allows us to use induction to reduce to the base case when there are no intermediate vertexes and so the result holds.
\end{proof}

\begin{figure}
\includegraphics{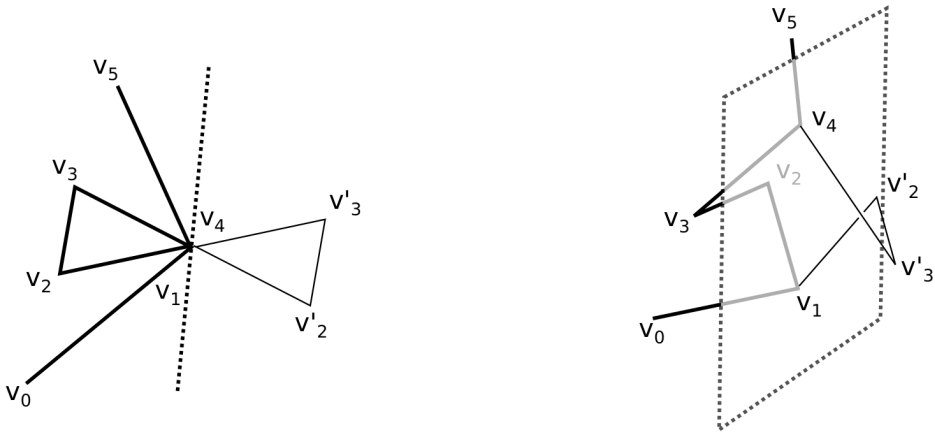}
\caption{An example of a reflection of an arc across a boundary plane. The left diagram is the projection into the plane.}
\label{fig:SampleReflection}
\end{figure}

So far we have shown that if our knot does not have an exposed projection, then there is an edge pair, and so we can reflect one arc connecting the edge pair across a common boundary plane to expands the projection. We now show that this expansion must have a limiting polygon.

\begin{lemma}
If \(K_i\) is a sequence of knots with \(K_{i+1} = K_{i}\) or \(K_{i+1} = r_i(K_i)\) where \(r_i\) is a reflection of an arc connecting an edge pair across a boundary plane, then the sequence \(K_i\) converges to a limit polygon \(K^*\).
\end{lemma}

\begin{proof}
Since the plane of reflection \(P_i\) is vertical, it projects to \(p(P_i) = l_i\). Reflection of an arc across a line which does not intersect the interior of the convex hull does not decrease the distance between pairs of points. Thus, each distance between pairs \(p(v_i)\) and \(p(v_j)\) is increasing, but it is also bounded above by half of the arclength of the knot. Thus, each distance between pairs of vertex projections must converge.

We now consider two cases. First suppose there is a collection of three vertex's projections, \(p(v_i),p(v_j),p(v_k)\) whose pairwise distances converge to a strict triangle inequality. Then there exists an \(N\) so that for all \(n \geq N\), \(K_n\) has \(p(v_i),p(v_j),p(v_k)\) satisfying a strict triangle inequality, meaning those three vertex projections are not collinear. Thus, since the location of a point in a plane can be determined by its distance to any three non-collinear points, then the position of every other vertex projection converges. Since none of the heights were changed, the location of every vertex of \(K_i\) converges and so there is a limit polygon \(K^*\).

In the second case, every triple of vertex projections \(p(v_i),p(v_j),p(v_k)\) have distances which converge to a triangle equality. This means every triple of projections converge to collinearity, since only collinear points satisfy triangle equality. Thus, taking a pair \(p(v_i)\) and \(p(v_j)\) and the line through them, every other vertex must converge to a specific point on that line. Thus, the location of the vertexes projections converge and the heights of the vertexes were unchanged so the location of the vertexes converge giving a limit polygon \(K^*\).
\end{proof}

We now know that a limiting polygon exists. We use information about that limit to get information about elements of the sequence. We first show that vertexes which are exposed in the limit are actually exposed in finite time.

\begin{lemma}
\label{lem:ExposedFiniteTime}
If \(K_i\) is a sequence of knots with \(K_{i+1} = K_{i}\) or \(K_{i+1} = r_i(K_i)\) where \(r_i\) is a reflection of an arc connecting an edge pair across a boundary plane and \(v_j\) is a vertex of the knots which has exposed projection in the limit \(K^*\), then there exists an \(N\) with \(v_j\) having exposed projection in \(K_n\) for all \(n \geq N\).
\end{lemma}

\begin{figure}
\includegraphics{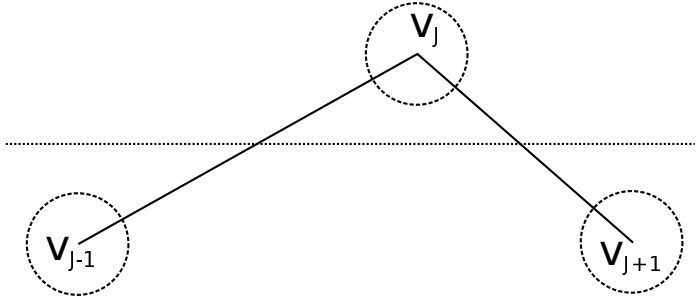}
\caption{A model of a triple of points near a vertex which is exposed in the limit, along with a line separating this exposed vertex from the rest of the knot}
\label{fig:Exposed}
\end{figure}

\begin{proof}
Because \(p(v_j)\) is exposed in the limit, it is a vertex of \(conv(p(K^*))\) so there is a line \(l\) which separates the limit of \(p(v_j)\) from the limit of all other vertex projections, as in figure~\ref{fig:Exposed}. This line \(l\) separates the plane into two open sets. The convergence of each point, means that for each vertex \(v_k\) there is a natural number \(N_k\) with \(p(v_k)\) on the correct side of \(l\) in all \(K_n\) with \(n \geq N_k\). Taking \(N = max(N_k)\), we get that for all \(n \geq N\), \(K_n\) has all vertex projections on the same side of \(l\) as \(K^*\). This means that for all \(K_n\) with \(n \geq N\), there is a line separating \(p(v_j)\) from all other vertexes, so \(p(v_j)\) is an exposed vertex of \(p(K_n)\).
\end{proof}

So far we have allowed any choice of reflection of arc connecting an edge pair. We will now show that if we make a specific choice of reflection, then the limit will be exposed. We define \(\mu(p(K))\) to be the sum over pairs of vertexes \(v_i,v_j\) of the distance between \(v_i\) and \(v_j\).

\begin{lemma}
\label{lem:LimitExposed}
Suppose \(K_i\) is a sequence of knots with \(K_{i+1} = K_{i}\) or \(K_{i+1} = r_i(K_i)\) where \(r_i\) is a reflection of an arc connecting an edge pair across a boundary plane, with \(K_{i+1}\) chosen so that the number of vertexes \(v_j\) with turning angle at \(p(v_j) = \pi\) is reduced if possible, and otherwise chosen so that \(\mu(p(K_{i+1}))\) is maximized. Then, the limit polygon \(K^*\) is exposed.
\end{lemma}

\begin{proof}
Seeking a contradiction, suppose \(K^*\) has a projection which is not exposed. This means there is a reflection move on \(K^*\). This involves increasing \(\mu(p(K^*))\) by an amount \(\delta\). Taking the convex hull and the sums of distances between vertex projections are both continuous operations which means the maximum change in \(\mu(p(K_i))\) must converge to a value greater than or equal to \(\delta\). Our choice of reflections, however, ensures that the maximum change in \(\mu(p(K_i))\) converges to zero. This is a contradiction and so the limit \(K^*\) must be exposed.
\end{proof}

We now know that if chosen properly, then the reflection moves will have an exposed limit. We next show that the limit is reached in finite time, first by proving it when the limit is full dimensional and then again when the limit is subdimensional.

\begin{lemma}
\label{lem:FullDimTerminates}
Suppose \(K_i\) is a sequence of knots with \(K_{i+1} = K_{i}\) or \(K_{i+1} = r_i(K_i)\) where \(r_i\) is a reflection of an arc connecting an edge pair across a boundary plane, with \(K_{i+1}\) chosen so that the number of vertexes \(v_j\) with turning angle at \(p(v_j) = \pi\) is reduced if possible, and otherwise chosen so that \(\mu(p(K_{i+1}))\) is maximized. If the projection of the limit polygon \(p(K^*)\) is full dimensional, then there exists an \(N\) so that for any \(n > N\), \(K_{n+1} = K_n\).
\end{lemma}

\begin{proof}
Lemma~\ref{lem:LimitExposed} tells us that the projection of the limit \(p(K^*)\) is an exposed polygon. The fact that \(K^*\) is an exposed polygon tells us that \(K^* = \bigcup_{l=0}^m A_l\), where \(A_l\) is an arc connecting adjacent exposed vertexes of \(K^*\). We also know that since \(K^*\) is full dimensional, for any pair of vertexes \(v_{l_1}, v_{l_2}\) on the interior of distinct arcs \(A_{l_1}, A_{l_2}\), the line segment connecting \(p(v_{l_1}), p(v_{l_2})\) intersects the interior of the convex hull of \(p(K^*)\). The continuity of taking the convex hull and taking convex combinations of vertexes tells us that there is a natural number \(N_{l_1,l_2}\) so that for every \(n \geq N_{l_1,l_2}\), \(K_n\) has the line connecting \(v_{l_1},v_{l_2}\) intersects the interior of the convex hull of \(p(K_n)\). Taking \(N\) to be the maximum over this finite number of bounds, we get that for every \(n \geq N\), any line connecting vertexes on distinct arcs \(A_{l_1}, A_{l_2}\) intersects the interior of the convex hull of \(p(K_n)\), and so every edge pair lies in one of the arcs \(\{A_l\}_{l=1}^m\).

The fact that every edge pair lies in one of the arcs \(\{A_l\}_{l=1}^m\) tells us that any reflection of arc connecting an edge pair must have either the reflected arc or the stationary arc containing all vertexes which are exposed in \(p(K^*)\). Without loss of generality, we can assume that the stationary arc is the one which contains all vertexes exposed in \(p(K^*)\). Thus, every vertex which is exposed in \(p(K^*)\) is fixed for all \(n > N\). For any vertex \(v_j\) on the interior of one of the arc \(A_l\), we have the triple of vertexes \(v_j, v_{l_0}, v_{l_k}\), where \(v_{l_0}, v_{l_k}\) are the endpoints of the arc \(A_l\). The distance \(d_1 = d(p(v_{l_0}),p(v_{l_k}))\) is fixed, the distances \(d_2 = d(p(v_{l_0}),p(v_j))\) and \(d_3 = d(p(v_j),p(v_{l_k}))\) are non-decreasing, and the limit satisfies \(d_1 = d_2 + d_3\) which is the minimum for \(d_2+d_3\), so for every \(n \geq N\), \(d_1 = d_2+d_3\) which uniquely determines the location of \(v_j\). Thus, for every \(n \geq N\), every vertex is fixed. Thus, \(K_{n+1} = K_n\).
\end{proof}

\begin{lemma}
\label{SubDimTerminates}
Suppose \(K_i\) is a sequence of knots with \(K_{i+1} = K_{i}\) or \(K_{i+1} = r_i(K_i)\) where \(r_i\) is a reflection of an arc connecting an edge pair across a boundary plane, with \(K_{i+1}\) chosen so that the number of vertexes \(v_j\) with turning angle at \(p(v_j) = \pi\) is reduced if possible, and otherwise chosen so that \(\mu(p(K_{i+1}))\) is maximized. If the projection of the limit polygon \(p(K^*)\) is subdimensional, then there exists an \(N\) so that for any \(n > N\), \(K_{n+1} = K_n\)
\end{lemma}

\begin{proof}
Lemma~\ref{lem:LimitExposed} tells us that the projection of the limit \(p(K^*)\) is an exposed polygon. There can only be finitely many \(r_i\) which reduce the number of turning angle \(\pi\) vertexes so there is a number \(M\) so that for every \(n > M\), no such reflection is possible. The fact that \(p(K^*)\) is subdimensional means that each exposed vertex  \(p(v_i)\) has turning angle \(\pi\), so both edges coming out from \(p(v_i)\) go in the same direction. By Lemma~\ref{lem:ExposedFiniteTime}, there is an \(N\) with \(p(v_i)\) exposed in \(K_n\) for every \(n \geq N\). If \(n > max(N,M)\) and \(p(K_n)\) is not subdimensional, then there are two distinct edges of \(conv(p(K_n))\) coming from \(p(v_i)\), so at least one supporting line \(l\) does not contain the edges coming out of \(p(v_i)\). This means that there is an edge pair made of \(v_i\) and another vertex \(v_j\) with \(p(v_j)\) on \(l\). Reflecting an arc connecting \(v_i\) to \(v_j\) reduces the number of vertexes with turning angle \(\pi\). This contradicts the choice of \(M\) so we know that for all \(n > max(M,N)\), \(p(K_n)\) is subdimensional. Since a subdimensional polygon has exactly its endpoints as exposed vertexes, every reflection move between subdimensional polygons changes which two vertexes are exposed, and for every \(n > N\) the two exposed vertexes of \(p(K^*)\) are exposed vertexes of \(p(K_n)\), no reflection can be applied to \(K_n\). Thus, for every \(n > max(N,M)\), \(K_{n+1} = K_n\).
\end{proof}

\begin{theorem}
\label{thm:ConvexifyPlanar}
For any knot \(K\), there is a finite number of reflections \(r_i\) of arcs connecting edge pairs across boundary planes, generating a sequence \(K = K_0,\ldots,K_n\), with \(r_i(K_i) = K_{i+1}\) and \(p(K_n)\) an exposed polygon.
\end{theorem}

\begin{proof}
By Lemma~\ref{lem:CommonEdge} there is a reflection as long as the result is not yet exposed. By lemmas~\ref{lem:FullDimTerminates} and ~\ref{SubDimTerminates} we know that if those reflections are chosen correctly, only finitely many are required to reach the limit polygon, which by Lemma~\ref{lem:LimitExposed} is an exposed polygon.
\end{proof}

\subsubsection{Moving From Exposed to Convex}

To reach a convex polygon, we need to make the number of unnecessary intersections of edge projections go to zero. We measure this distance from being an embedding by taking the \emph{incidence} of the knot to be 
\begin{align*} I(K) = |\{(e_i,e_j) \text{ such that } p(e_i)\cap p(e_j) \text{ non-trivial }\}|\end{align*}
Using the the previous section we can get the following theorem about polygonal knots with thickness.

\begin{theorem}
\label{thm:Convexify}
Any knot \(K\) can be transformed into a knot \(K'\) whose orthogonal projection into the \(x-y\) plane is an exposed polygon using a finite number of height preserving, thickness non-decreasing reflections \(r_i\), with \(I(K') \leq I(K)\) and each \(r_i\) in the closure of the interior of moves which do not decrease thickness.
\end{theorem}

\begin{proof}
We first take the standard orthogonal projection of \(K\) into the \(x-y\) plane. This projection gives a polygonal loop, \(p(K)\), with possible intersections. We can apply Theorem~\ref{thm:ConvexifyPlanar} to find a sequence of finitely many reflections \(r_i\) of arcs connecting edge pairs across boundary planes \(P_i\) which take \(K\) to \(K'\) with \(p(K')\) exposed.

Each of these reflection moves is across a plane parallel to the \(z\)-axis, and so the \(z\) coordinate is unchanged, making the process height invariant. We also are reflecting at each step across a plane which does not intersect the convex hull of the current knot, which means that by Lemma~\ref{lem:Thickness} the thickness is not decreased by any reflection. The fact that the reflection is of an arc connecting an edge pair means that \(P_i\) lies in a positive length interval of planes which contain the same edge pair but do not intersect the interior of the convex hull. Thus, \(r_i\) is in the closure of the interior of moves which do not decrease thickness. Finally, if we have two edges whose projections intersect non-trivially after a reflection move, then either both were fixed, both moved, or exactly one moved. If both moved or both were fixed, then the non-trivial intersection was present before the reflection. If one moved and the other didn't, then they each are placed on opposite sides of \(p(P_i)\). Thus, their intersection lies on \(p(P_i)\) and was fixed by \(r_i\) meaning it was present before the reflection. Therefore there can be no newly created incident edge pairs, and so \(I(K') \leq I(K)\). Thus, all the properties in the theorem are shown.
\end{proof}

Our next step is to find reflections which decrease the incidence of the knot projection, without increasing thickness or changing the set of minimal height vertexes. We know that we have injectivity at the vertexes of the convex hull, so we turn our attention to the edges. Whenever injectivity fails inside an edge, we will push a part of the preimage of that edge out slightly to remove some of the incidences and continue the process.

\begin{lemma}
\label{lem:Pushout}
If \(K\) is a knot with \(p(K)\) an exposed polygon which is not convex, there exists a reflection move or affine transformation \(r\) with \(I(r(K)) < I(K)\), \(r(K)\) having no fewer minimal height vertexes than \(K\), and \(r\) in the interior of moves which do not decrease thickness.
\end{lemma}

\begin{proof}
We first consider the special case where \(p(K)\) is subdimensional. In that case, \(K\) lies in a vertical plane, so we can apply an affine transformation \(r\) moving this plane to horizontal. This means \(p(r(K))\) is an embedding, so \(I(r(K))\) is zero, \(r(K)\) has the same thickness as \(K\), and \(r(K)\) has the maximum number of minimal height vertexes.

We next move to the case where \(p(K)\) is full dimensional. Note that if \(p(K)\) is an exposed polygon which is not convex, then there is an edge \(e\) of the convex hull of \(p(K)\) for which \(p|_{p^{-1}(e)}\) is not an injection. The set \(p^{-1}(e)\) is an infinite strip, with \(K\cap p^{-1}(e)\) going from one side of the strip to the other, as in figure~\ref{fig:strip}. We will denote this arc \(A\) for simplicity. The fact that \(A\) moves from one side of the strip to the other tells us that the convex hull of \(A\) separates the strip into a top and bottom portion. In particular, the doubling back means that \(A\) doesn't live inside the lower boundary of its convex hull. Thus, there is a line \(l\) defining the bottom portion of the convex hull of \(A\) and a subarc \(a\) with both ends on \(l\) but whose interior misses \(l\). We can assert this by choosing a point \(\alpha\) which lies directly above another, and proceeding forward and back until we reach the lower boundary of the convex hull of \(A\) at the points \(\alpha_+\) and \(\alpha_-\). The knot being injective means that the arc connecting \(\alpha_-\) to \(\alpha_+\) through \(\alpha\) separates the convex hull, and so anything on the line segment connecting \(\alpha_-\) to \(\alpha_+\) would be separated from the endpoints of \(A\). We will reflect across a plane containing \(l\).

\begin{figure}
\includegraphics{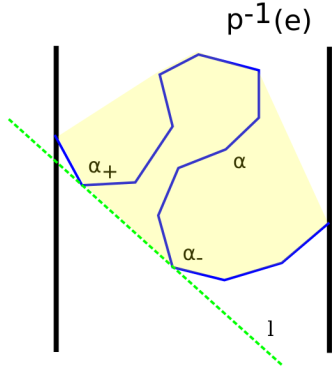}
\caption{An example of \(p^{-1}(e)\) with the convex hull of \(A\) highlighted}
\label{fig:strip}
\end{figure}

If we choose the plane of reflection to be perfectly vertical, then we have accomplished nothing, but if the plane of reflection is too shallow, we may create problems for thickness or minimal height vertexes. We will consider a cylindrical coordinate system, where the center of the cylinder is the line \(l\) about which we are rotating. We orient this coordinate system so the angle \(0\) corresponds to the half plane directly below \(l\) and the angle \(\pi\) to the half plane above \(l\). The fact that \(l\) lies on the boundary of the convex hull means we can assert that the entire knot has angles in \([0,\pi]\), and the fact that \(l\) is a lower boundary of the convex hull of \(A\) means that all angles are greater than \(0\). The knot \(K\) is compact, which means it attains its minimal angle \(\theta_{min}\). This means that if we choose a plane through \(l\) at angle \(0 < \epsilon < \theta_{min}\), this plane intersects the convex hull of \(K\) only in the line \(l\). In particular, Lemma~\ref{lem:Thickness} tells us that reflecting across this plane is a reflection move \(r\) with \(r(K)\) having no lower thickness than \(K\), as is any nearby reflection \(\tilde{r}\), so \(r\) is in the interior of moves which do not decrease thickness. The height of each vertex after the reflection is a continuous function of the angle \(\epsilon\), and so for a sufficiently small \(\epsilon\), any vertex which was above the minimal height will stay above the minimal height, and none at the minimal height are moved, so the set of minimal height vertexes is unchanged. Finally, we seek to show that \(I(r(K)) < I(K)\).

Any pair of edges which are both unmoved will have their incidences unchanged. Among those that are moved, we have at least some non-trivial intersections, by our choice of \(\alpha\). Further, the moved portion \(a\) is moved to a half plane \(h\) which is not vertical, and therefore \(p|_h\) is injective. Thus, there are no incidences among pairs of edges where at least one is moved. Thus, \(I(r(K)) < I(K)\) and so the result holds.
\end{proof}

This allows us to state the next main theorem.

\begin{theorem}
For any knot \(K\), there is a finite sequence of knots \(\{K_i\}_{i=0}^n\) and reflection moves \(r_i\) satisfying the following.
\begin{itemize}
\item \(K_0 = K\)
\item \(p(K_n)\) is convex
\item The set of minimal height vertexes is nondecreasing as \(i\) increases.
\item Each reflection move \(r_i\) is in the closure of the interior of moves which do not decrease thickness.
\end{itemize}
\end{theorem}

\begin{proof}
We induct on \(I(K)\). Regardless of the value of \(I(K)\), we can apply Theorem~\ref{thm:Convexify} and get a finite sequence which ends in an exposed polygon. If this polygon is convex, then we are done. Otherwise we apply Lemma~\ref{lem:Pushout}, which reduces the number of incidences, and so the inductive hypotheses completes the result.
\end{proof}

\subsection{Flattening the Knot}

In this section our goal is to produce a knot which is convex and planar using thickness non-decreasing moves. We will do so inductively using the number of vertexes at the minimum height. This will be done using two theorems.

\begin{theorem}
\label{thm:MoreMinVerticies}
If a knot has a convex orthogonal projection into the \(x-y\) plane, but is not planar, then there is a pair of reflection moves or a rotation that will increase the number of minimum height vertexes by at least one, with a neighborhood which does not decreasing the thickness.
\end{theorem}

\begin{proof}
There must be at least one vertex that attains the minimum height among vertexes. Consider first the case where there is exactly one vertex \(v\) which is at the minimum height. We can apply a rotation and change the minimum height vertex into a maximum height vertex. Intermediate value theorem tells us that for each other vertex, \(v'\), there is an intermediate rotation which makes the height of \(v\) equal the height of \(v'\). Since there are finitely many vertexes, there is a first vertex, \(w\) for which this happens. Since this is the first vertex to pass the height of \(v\), \(v\) must still be a minimum height vertex, and now \(w\) is also. Thus, we have increased the number of minimum height vertexes from one to at least two. As a rotation does not affect the knot thickness in any way, the result is shown in this case.

Now consider the case where there are at least two vertexes which attain the minimum height. Find a pair of such vertexes, \(v_1,v_n\), with the property that there is an arc of the knot connecting them with no minimum height vertexes, and with at least one vertex. The fact that the knot is not planar guarantees the existence of such an arc. Since the knot is convex in the projection, there is a vertical plane with the arc in the closure of one side and the complementary arc in the closure of the other side. We will use the fact that an arc may be rotated using a pair of reflection moves through a common axis. We can then use this to rotate the arc away from the separating plane, as in figure~\ref{fig:Rotation}. Each of the intermediate vertexes, \(\{v_i\}_2^{n-1}\) form an angle \(\theta_i\) with the line connecting \(v_1\) to \(v_n\), oriented so \(0\) represents a point lying at the height of \(v_1\) and \(v_n\), and \(\pi/2\) representing lying above the line connecting \(v_1\) to \(v_n\). There is a minimal such angel \(\theta_j\). Thus, if we rotate the arc connecting \(v_1\) to \(v_n\) down by an angle \(\theta_j\), we add \(v_j\) to the set of vertexes at the height of \(v_1\) and since \(\theta_j \leq \theta_i\) for \(2\leq i\leq n-1\), no vertex can have gone below \(v_1\) and so we have added a vertex to the set of minimum height vertexes.

\begin{figure}
\includegraphics{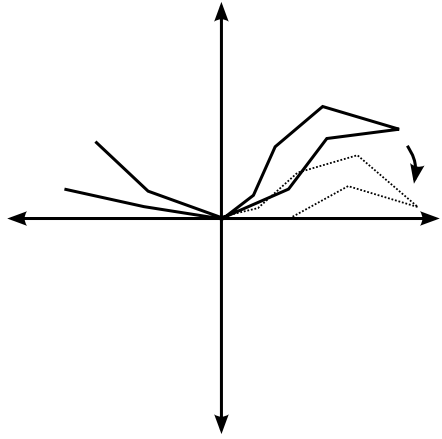}
\caption{An example of a rotation which increases the number of minimum height vertexes, shown from a projection which is parallel to the axis of rotation.}
\label{fig:Rotation}
\end{figure}

Consider two points in the realization, \(v, w\) and a rotation angle \(\theta\) with \(0\leq \theta \leq 2\theta_j\). We seek to apply Lemma~\ref{lem:distance} and so we would like to know that the distance between \(v,w\) is increased by the rotation, which means it suffices to show that their squared distance is increased. If both are unmoved by the rotation, or if both are rotated, then the distance between them is unchanged. Thus, we can suppose that \(v\) is moved, and \(w\) is fixed. We will consider cylindrical coordinates, with the axis of rotation the center of the cylinder. This causes the height and radius of both points to be unchanged. Further, since we are rotating about an axis which includes two vertexes which are absolute minimums, we can presume that the radii are both non-negative, that the fixed point, \(w\), has angle \(\phi_w\) between \(\pi\) and \(\pi/2\), and that the point being rotated, \(v\), has angle \(\phi_v\) which starts between \(\pi/2\) and \(\theta_j\). Let \(\phi\) be the difference between \(\phi_v\) and \(\phi_w\). The height coordinates are both constant, so the difference in the squared distance reduces to the following.

\begin{align}
(r_1\sin(\phi))^2 + (r_1\cos(\phi) - r_2)^2 &= r_1^2 - 2r_1r_2\cos(\phi) + r_2^2
\end{align}

Because \(\phi_v > \theta_j\), we know that \(\phi < \pi-\theta_j\) initially. Since \(\theta < 2\theta_j\) we can observe that \(|\pi - \phi|\) is reduced. This means that \(\cos(\phi)\) is decreased by the move, so the distance between the points is increased. This means that every rotation by an angle \(\theta\) with \(0 < \theta < 2\theta_j\) does not decrease thickness, so the particular rotation by \(\theta_j\) has a neighborhood which does not decrease thickness.
\end{proof}

We now combine theorems~\ref{thm:Convexify} and ~\ref{thm:MoreMinVerticies} for the following theorem.

\begin{theorem}
\label{thm:Planar}
Any knot may be made convex and planar using a sequence of finitely many moves consisting of affine transformations, reflections, or pairs of reflections, such that each move is contained in the closure of the interior of moves which will not increase the thickness.
\end{theorem}

\begin{proof}
We do this using induction on the number of vertexes which fail to attain the minimum height. If there are no vertexes which fail to attain the minimum height then all vertexes are in a common plane. Applying Theorem~\ref{thm:Convexify} gives a finite sequence of moves which makes it convex, does not change the height, and is in the closure of the interior of moves which do not decrease the thickness and so we are done. Now suppose that we can achieve the desired result if there are less than \(n\) vertexes which attain the minimum height, and that \(K\) is a knot which has exactly \(n\) vertexes which fail to attain the minimum height. We can then apply Theorem~\ref{thm:Convexify} to get a finite sequence of \(j\) reflections which do not affect the height, are in the closure of the interior of moves which do not decrease the thickness, and give us a knot configuration \(K_2\) which is convex in projection. Since the height is unchanged we still have \(n\) vertexes which fail to attain the minimum height. We can then use Theorem~\ref{thm:MoreMinVerticies} to get a pair of reflections or a rotation in the closure of the interior of moves which do not decrease thickness, which will give us a knot \(K_3\) with fewer than \(n\) minimum vertexes. Finally, we can use the inductive hypothesis to get a sequence of \(k\) moves in the closure of the interior of moves which do not decrease thickness which results in a convex planar polygon. Thus, \(K\) can be made convex and planar using \(j+k+1\) moves in the closure of the interior of moves which do not decrease thickness. Therefore, by induction, any knot can be made convex and planar using a finite sequence of moves in the closure of the interior of moves which do not decrease thickness.
\end{proof}

\subsection{Convex and Planar made Regular}

\subsubsection{The Move}

The move will be a sequence of six reflections, but the net result can be explained much more clearly, and visualized in figure~\ref{fig:Regularize}. We will take four vertexes and look at the planar knot as being the quadrilateral joined directly by those four vertexes, along with up to four flaps.

The result of the move is equivalent to allowing only the distinguished four vertexes to pivot in the plane, and then pushing two opposite vertexes together, allowing the complementary two to spread apart. This causes two interior angles to shrink and two interior angles to grow. We will pick two vertexes whose interior angles are large and two whose interior angles are small so applying this move minimally will make one of the four regular.

\begin{figure}
\includegraphics{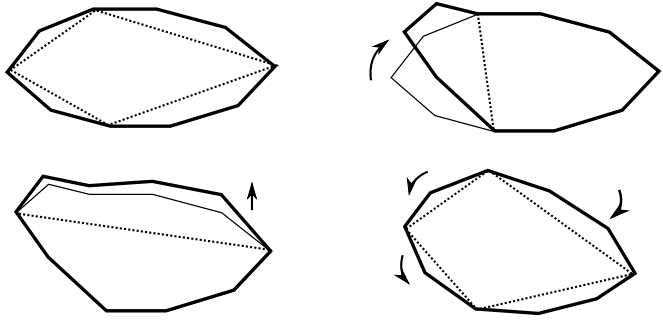}
\caption{An example of applying the hextuple reflection move to a decagon. The initial diagram has the four distinguished vertexes connected by dotted lines.}
\label{fig:Regularize}
\end{figure}

Informally, the way we achieve this type of move using reflections is that we reflect through a plane containing two of the flexible vertexes to bring the complementary vertexes closer together, and then use a reflection through the pair we moved to make the four planar. Finally we use up to four more reflections to bring the four flaps back into the common plane.

\subsubsection{Picking Vertexes}

As mentioned above, we find a collection of four vertexes, two which have large interior angle and two which have small interior angle, in an alternating pattern around the knot.

\begin{lemma}
\label{lem:VertexChoose}
For any convex non-regular polygon in the plane, there exists an ordered set of four vertexes, \(v_1,w_1,v_2,w_2\) with the interior angle of \(v_i\) smaller than regular, and the interior angle of \(w_i\) larger than regular.
\end{lemma}

\begin{proof}
For this, we first note that if the polygon is not regular, then it must have a vertex, \(w_1\), with an interior angle which is too big and a vertex, \(v_1\), whose interior angle is too small, since the sum of the angles is fixed. Without loss of generality, choose \(v_1\) and \(w_1\) to be separated only by vertexes with regular angles in at least one of the arcs. Next, we consider the arc consisting of just the two edges connected to \(w_1\). This arc has longer end to end distance than if the polygon was regular. This means that the complementary arc must also have longer end to end distance. Since we are dealing with a convex polygon, the end to end distance is an increasing function of the interior angles, and so there must be some second vertex \(w_2\) which has an interior angle which is too large. This vertex cannot be between \(v_1\) and \(w_1\) so without loss of generality, choose it to be the closest such vertex to \(v_1\) on the other side from \(w_1\). Thus, the arc connecting \(w_1\) to \(w_2\) through \(v_1\) consists only of vertexes with angles which are regular or smaller than regular. This means that this arc has shorter end-to-end distance than in the regular polygon, so the complementary arc has shorter end-to-end distance than the regular polygon, meaning it must also have a vertex, \(v_2\), which has an angle which is smaller than regular. Thus, we have a sequence of ordered vertexes \(v_1,w_1,v_2,w_2\), with \(v_1,v_2\) having angles which are smaller than regular and with \(w_1,w_2\) having angles which are bigger than regular.
\end{proof}

\subsubsection{Choosing Reflections}

We will assume the knot lies in the \(x-y\) plane and form a transformation of the knot \(T\), consisting of up to six reflections which result in \(T(K)\) begin planar and having more regular vertexes than \(K\). The first reflection will reflect the arc connecting \(w_1\) to \(w_2\) through \(v_1\) across a plane which makes an angle of \(\theta\) to the \(x-y\) plane. There are generally two such choices so we will choose the one which puts the vertex \(v_1\) to a non-negative \(z\) coordinate. The second will reflect the arc connecting \(v_1\) to \(v_2\) through \(w_2\), and will be chosen in such a way as to make the four vertexes \(v_1,v_2,w_1,w_2\) coplanar. There are generally two such choices so choose the one which maximizes the \(z\) coordinate of \(w_2\). Finally, the four arcs connecting the vertexes of our highlighted quadrilateral may be out of alignment, so we will reflect them each across a plane which will bring them into the common plane of \(v_1,v_2,w_1,w_2\). Again there are two such choices so we will choose the one which moves each flap the shortest distance. This gives a family of transformations, \(T_\theta\), for each real number, \(\theta\) between \(0\) and \(\pi/2\), with \(T_{\theta}(K)\) continuous in \(\theta\). This leads us to the next result.

\begin{theorem}
\label{thm:Regularize}
Any convex planar knot can be made regular using a sequence of finitely many moves, each move consisting of at most six reflections, with each such move in the interior of moves which do not decrease thickness.
\end{theorem}

\begin{proof}
We will show that there is a choice of \(\theta\) for which \(T_\theta(K)\) has exactly one of the four vertexes have a regular interior angle without switching any of the other four between larger than regular interior angle and smaller than regular interior angle. For this, we consider the vectors \(e_1 = w_1-v_1\), \(e_2 = v_2-w_1\), \(e_3 = w_2-v_2\), and \(e_4 = v_2-w_1\). These are the four edges of the inscribed quadrilateral formed by those four vertexes. We then look at the function \(f(K) = (e_1 \times e_2) \cdot (e_3 \times e_4)\). For \(\theta = 0\), \(f(T_\theta(K)) >0\). For \(\theta = \pi/2\), \(f(T_\theta(K)) < 0\). Thus, since \(f\) is a continuous function and \(T_\theta\) is a continuous function of \(\theta\), we use the intermediate value theorem to observe that there is a choice \(\theta_0\) for which \(f(T_{\theta_0}(K)) = 0\), so \(e_1\times e_2\) and \(e_3\times e_4\) are perpendicular vectors.. The configuration \(T_\theta(K)\) is always planar, so the vectors \(e_1\times e_2\) and \(e_3 \times e_4\) are parallel. This means that the two vectors are both parallel and perpendicular, which can only be the case if one of them is the zero vector. Thus, either \(e_1\times e_2\) is zero or \(e_3\times e_4\) is zero, so in \(T_{\theta_0}(K)\), one of the \(v_i\) is collinear with \(w_1\) and \(w_2\).

For the remainder of the proof \(\theta \leq \theta_0\). This means that in particular \(T_\theta(K)\) is still an embedding. Thus, we have a well defined notion of the sign of a turning angle, where the sign of all turning angles in \(T_0(K)\) are positive. The fact that the interior angle of \(v_i\) in \(T_0(K)\) is smaller than regular tells us the turning angle at \(v_i\) in \(T_0(K)\) is larger than regular. Because \(v_i\) is collinear with \(w_1\) and \(w_2\) in \(T_{\theta_0}(K)\), the turning angle of \(v_i\) in \(T_{\theta_0}(K)\) is negative. This means that by the intermediate value theorem, there exists an angle \(\theta_1\) with the turning angle of \(v_i\) in \(T_{\theta_1}(K)\) regular.

There may be multiple choices of \(\theta\) with \(T_\theta(K)\) having a greater number of regular angles than \(K\). We then choose the smallest \(\theta\) which makes at least one of the four changing angles regular. Thus, one particular angle is made regular, and the remainder stay on the same side of regular, either staying larger than regular or staying smaller than regular. We note that from corollary~\ref{cor:PlanarThickness}, the long range thickness need not be considered. One possibility is that the smallest angle is unchanged, which means that for any \(T\) in a neighborhood of \(T_\theta\), the thickness of \(T(K)\) is the same as the thickness of \(K\). The other possibility is that the smallest angle is one of the \(v_i\) and so is larger, meaning that \(T_\theta(K)\) is thicker than \(K\) and so since thickness is continuous, a neighborhood of \(T_\theta\) does not decrease thickness. Therefore we have increased the number of regular vertexes using a move on the interior of moves which do not decrease thickness. Applying this up to \(n\) times, where \(n\) is the number of vertexes, we can guarantee that the end result has at least \(n\) regular vertexes, and so is the regular planar polygon.
\end{proof}

\subsection{Result}

\begin{theorem}
\label{thm:FullConnected}
Any equilateral polygonal knot can be made into a regular convex planar polygon using a finite number moves, consisting only of rigid motions, or up to six reflections, with each move in the closure of the interior of moves which do not decrease thickness.
\end{theorem}

\begin{proof}
This is a simple combination of the above theorems. First, Theorem~\ref{thm:Planar} shows that a finite number of such moves can make any equilateral polygonal knot planar and convex, and then Theorem~\ref{thm:Regularize} allows another finite number of such moves to make any equilateral polygonal planar and convex knot regular. This gives us the desired result.
\end{proof}

\section{Monte-Carlo Markov Chains}

In this section, we will prove that we can build an ergodic Markov chain for sampling the space of equilateral knots with thickness. The above theorem will be necessary in showing ergodicity.

\subsection{The Markov Chain}

The general definition of a Markov chain is a sequence in a state space \(X\) where each entry in the sequence of states is independent of every state before the immediate predecessor. We will refer to our Markov chain as \(\Phi\). Our Markov chain is given by a set of states \(X\) in \(\R^n\), a noise parameter \(W\) which is an open subset of \(\R^p\) with a probability measure \(\mu_W\), and a function \(F:X\times W \rightarrow X\). The space of states will be \(X = \Equ(n,t)\). The noise parameter \(W\) will be a product of reflection moves. The space of reflection moves \(R\) is of the form \((x,v_1,v_2,\theta)\), where \(x\) is a real number in \((0,1)\) used to determine if we check the thickness, \(v_1,v_2\) are a random pair of vertexes, and \(\theta\) is an angle. While this \(R\) isn't clearly an open subset of \(\R^p\), it can be modeled in that way by choosing the continuous variables from a rectangle, and selecting the discreet variables by using a disjoint collection of rectangles. Thus, we use 
\begin{align}
R = \bigcup_{0\leq i < j \leq n} (0,2\pi)\times (2(n*i+j),2(n*i+j)+1)
\end{align}
where \(n\) is the number of edges. This space \(R\) is a disjoint collection of rectangles. The \(y\) coordinate is an angle. The width of each rectangle is \(1\). Which rectangle a point is in determines the \(i,j\) which represents a pair of distinct vertexes. The complete noise space is \(W = R^N\), where \(N\) is a cap on how many reflections we are allowed between checking the thickness. We require that \(N \geq 6\) so that we may use Theorem~\ref{thm:FullConnected}. The probability measure on \(W\) is just a multiple of the Lebesgue measure on this bounded space, which means that it has a constant probability distribution \(\gamma_W\) which is supported on all of \(W\) and is lower semi-continuous since \(W\) is open.

The function F in our Markov chain \(\Phi\) will be defined as follows. We have a collection of \(N\) probabilities \(p_k\) which represent the probability of applying the \(k^{th}\) reflection if we already applied the \(k-1^{st}\) reflection. Given a \(w \in W = R^N\) we find a sequence of variables \(a_k = x_k - \lfloor x_k\rfloor\), where \(x_k\) is the \(x\) coordinate of \(W_k\cong R\). These \(a_k\) live in \((0,1)\). We find \(m\) which is one less than the first \(k\) with \(a_k > p_k\), or \(m=N\) if no such \(k\) exists. This \(m\) is how many reflections we will apply. When we apply a reflection to a knot \(K\in\Equ(n,t)\), we reflect the arc connecting the vertexes \(v_i\) and \(v_j\) across a plane determined by the angle \(\theta\), where \(i,j\) are indicated by which component of \(W_k \cong R\) we are in and \(\theta = y_k\) from \(W_k \cong R\). Thus, we get a reflection move \(r_k\) for each \(k\). This lets us define \(F(K,w) = r_m r_{m-1}\ldots r_1 (K)\) if \(r_m r_{m-1}\ldots r_1 (K) \in \Equ(n,t)\) and \(K\) otherwise. This means \(F\) is a piecewise function with two pieces, one is the identity and the other is smooth.

\subsection{Forward Accessible}
We will use \(A_+^n(x) \subseteq X\) to denote the set of states \(y\) for which there is a sequence of exactly \(n\) moves on the interior of a smooth section starting at \(x\) and ending at \(y\). We will also define \(A_+(x) = \bigcup_{n=1}^{\infty} A^n(x)\). One useful property a Markov chain can have is forward accessibility. A Markov chain is forward accessible if \(A_+(x)\) is a set with non-empty interior, for every value of \(x\). When the motion of the Markov chain is smooth, then the set \(A_+^n(x)\) reduces to simply points reachable in \(n\) steps, but since our function is merely piecewise smooth we need the more general definition. We will prove our Markov chain is forward accessible. This will be a direct continuation of the conclusion of the previous chapter.

\begin{lemma}
\label{lem:Reachable}
For every knot \(K\), there is a finite sequence of moves \(m_i\), \(0\leq i< N\) giving a finite sequence of knots \(K_i\), such that \(K_0 = K\), \(K_N\) is the regular planar polygon, and each \(m_i\) has a neighborhood of moves \(M_i\), with \(m(K_i) \in \Equ(n,t)\) for every \(m\in M_i\).
\end{lemma}

\begin{proof}
By Theorem~\ref{thm:FullConnected}, there is a sequence of \(N_1\) moves \(m^1_i\) with \(K^1_0 = K\), \(K^1_{N_1}\) the regular planar polygon, and each move \(m^1_i\) in the closure of the interior of moves which do not decrease thickness. This means that for any open neighborhood of the regular planar polygon, such as int(\(\Equ(n,t)\)), we can take a nearby sequence of \(N_1\) moves \(m^2_i\) with \(K^2_0 = K\), \(K^2_{N_1}\) in the open set int(\(\Equ(n,t)\)), and each \(m^2_i\) on the interior of moves which do not decrease thickness, and so in particular has a neighborhood of moves \(M^2_i\) with \(m(K_i)\in\Equ(n,t)\) for every \(m\in M^2_i\). We also can use the same theorem to generate a sequence of \(N_2\) moves \(m^3_i\) with \(K^3_0 = K^2_{N_1}\), \(K^3_{N_2}\) the regular planar polygon and each move \(m^3_i\) not decreasing thickness. Since each \(K^3\) is on the interior of \(\Equ(n,t)\), there is a neighborhood \(M^3_i\) of each \(m^3_i\) with \(m(K^3_i) \in\Equ(n,t)\) for every \(m\in M^3_i\). Concatenating these two sequences of moves, we get the desired result.
\end{proof}

\begin{theorem}
\label{thm:Accessible}
The Markov chain \(\Phi\) is forward accessible with int\((\Equ(n,t))\subseteq A_+(K)\) for every knot \(K\) in \(\Equ(n,t)\).
\end{theorem}

\begin{proof}
In short we will use lemma~\ref{lem:Reachable} to form a sequence connecting two knots by way of the regular planar polygon.	

Let \(K\in\Equ(n,t)\) and \(K'\in\) int(\(\Equ(n,t)\)). Then \(K' \in\Equ(n,t')\) where \(t' > t\). By Lemma~\ref{lem:Reachable}, there is a sequence of \(N_1\) moves \(m^1_i\) with \(K^1_0 = K\), \(K_{N_1}\) the regular planar polygon, and each \(m^1_i\) in an open neighborhood \(M^1_i\) with \(m(K_i) \in\Equ(n,t)\) for all \(m \in M_i\). By the same lemma there is also a sequence of \(N_2\) moves \(m^2_i\) with \(K^2_0 = K'\), \(K^2_{N_2}\) the regular planar polygon, and each \(K^2_i\) in \(\Equ(n,t') \subseteq\) int(\(\Equ(n,t)\)). Each reflection move can be repeated on the image knot, which means that each move \(m^2_i\) has an inverse move \(m^3_i := (m^2)^{-1}_{N_2-i}\). Thus, \(K^3_0\) is the regular planar polygon, \(K^3_{N_2} = K'\), and each \(K^3_i\) is in int(\(\Equ(n,t)\)). Since each \(m^3_i\) has an image in the interior of \(\Equ(n,t)\), there is a neighborhood of each, \(M^3_i\) with \(m(K^3_i)\) in \(\Equ(n,t)\) for every \(m\in M^3_i\). Concatenating these two sequences of moves gives a sequence of \(N\) moves \(m^4_i\) with \(K^4_0 = K\), \(K^4_N = K'\), and each move \(m^4_i\) having a neighborhood of moves \(M^4_i\) with \(m(\hat{K}_i)\in\Equ(n,t)\) for all \(m\in M^4_i\). This tells us that the sequence of moves \(m^4_i\) is on the interior of a smooth section, and so its image \(K'\) is in \(A_+(K)\).
\end{proof}

\subsection{T-Chains}
A very useful property for a Markov chain to have is being a \(T\)-chain. The precise definition requires a couple of preliminaries, but we will define it precisely in this section. To motivate it, a \(T\)-chain has moves of the Markov chain respecting the topology of the state space. We first define the \emph{transition probabilities} \(P^n(x,S)\) which is the probability of starting at \(x\), and ending in \(S\) after \(n\) steps. These are examples of \emph{transition kernels} which are functions \(\kappa: X\times \mathcal{B}(X) \rightarrow [0,1]\) with \(\mathcal{B}(X)\) the Borel sets. A transition kernel \(\kappa\) is \emph{stochastic} if \(\kappa(x,X) = 1\) for all \(x\in X\) and is \emph{substochastic} if \(\kappa(x,X) \leq 1\) for all \(x\in X\). This allows us to note that \(P^n(x,X)\) is the probability of staying in the total state space after \(n\) steps, and so these standard probability transition functions are stochastic. From the standard transition probability and a sampling on the natural numbers \(a:\N\rightarrow [0,1]\), we can build the stochastic transition kernel \(K_a(x, S) := \sum_{n=0}^{\infty} P^n(x,S) a(n)\).

This finally allows us to define a \emph{\(T\)-chain} as a Markov chain for which there is a substochastic transition kernel \(T\) and a sampling distribution \(a\) with \(K_a(x,S) \geq T(x,S)\), \(T(\cdot,S)\) lower semicontinuous, and \(T(x,X) > 0 \) for every \(x\). The lower semicontinuity is what lets us refer to \(T\) as the continuous piece of \(K_a\), while \(T(x,X) > 0\) for every \(x\) tells us that there is a significant continuous piece of \(K_a\). With these we utilize Propositions 7.1.5 and 6.2.4 from Meyn and Tweedie's book\cite{MeynTweedie93}.

\begin{theorem}
\label{thm:TKernals}
If a Markov chain which is forward accessible has a noise space \(W\) with a lower semi continuous probability density function \(\gamma_W\), then for each \(x\) in the state space \(X\), there is a sampling distribution \(a_x:\N\rightarrow [0,1]\) and a transition kernel \(T_x\) with \(T_x(x,X) \neq 0\), \(T_x(x,\cdot) \leq K_{a_x}(x,\cdot)\) and \(T_x(\cdot,S)\) lower semi-continuous for every \(S\).
\end{theorem}

This theorem amounts to taking a neighborhood of the smooth path that starts at \(x\) and ends in an open set, and using the implicit function theorem to pull that information back into a neighborhood of \(x\).

\begin{theorem}
\label{thm:KernelToChain}
Suppose a Markov chain whose state space \(X\) is a subset of \(\R^n\) and has for every \(x\in X\), there is a transition kernel \(T_x\) and sampling distribution \(a_x:\N\rightarrow [0,1]\) with \(T_x(x,X) \neq 0\), \(T_x(x,\cdot) \leq K_{a_x}(x,\cdot)\) and \(T_x(\cdot,S)\) lower semi-continuous for every \(S\). Then the Markov chain is a \(T\)-chain
\end{theorem}

\begin{proof}
For each \(x\in X\) there is a transition kernel \(T_x\) given by the hypothesis of the theorem. These kernels have sets \(O_x = \{y\in X| T_x(y,X) > 0\}\) which are open since \(T_x(\cdot,X)\) is lower semi-continuous. By assumption, \(x\in O_x\) so \(\{O_x\}_{x\in X}\) forms an open cover. By Lindel\"{o}f's theorem, this open cover has a countable subcover \(\{O_i\}_{i\in\N}\). This countable subcover corresponds to a countable set of states \(\{x_i\}_{i\in\N}\) with corresponding transition kernels \(T_i\) and sampling distributions \(a_i\). Let \(T = \sum_{i\in\N} 2^{-i} T_i\) and \(a = \sum_{i\in\N} 2^{-i} a_i\). This tells us the following.

\begin{align}
T &= \sum_{i\in\N} 2^{-i} T_i\\
&\leq \sum_{i\in\N} 2^{-i} K_{a_i}\\
&= \sum_{i\in\N} 2^{-i} \sum_{j\in\N} a_i(j) P^j\\
&= \sum_{j\in\N}P^j \sum_{i\in\N}2^{-i}a_i(j)\\
&= \sum_{j\in\N}P^j a(j) = K_a
\end{align}

Further, \(T(x,X) > 0\) for every \(x\) since the \(O_i\) form an open cover and \(T(\cdot,S)\) is lower semi-continuous since each \(T_i\) is lower semi-continuous and the series converges uniformly.
\end{proof}

We can combine these two results to get the following.

\begin{cor}
The Markov chain \(\Phi\) is a \(T\)-chain.
\end{cor}

\begin{proof}
By Theorem~\ref{thm:TKernals} and Theorem~\ref{thm:KernelToChain}, \(\Phi\) is a \(T\)-chain since it is forward accessible and has a density function which is supported and constant on an open set, and so in particular is lower semi-continuous.
\end{proof}

A sequence of probability distributions \(\mu_k\) is \emph{tight} if for every \(\epsilon > 0\), there is a compact set \(C\) with \(\liminf(\mu_k(C)) > 1-\epsilon\). 

\begin{lemma}
Any sequence of probability distributions on a compact space is tight.
\end{lemma}

\begin{proof}
Let \(\mu_k\) be a sequence of probability distributions on a compact space \(C\). Since they are probability distributions, \(\mu_k(C) = 1\) which means that \(\liminf(\mu_k(C)) = 1 > 1-\epsilon\) for every \(\epsilon > 0\).
\end{proof}

A Markov chain is \emph{bounded in probability on average} if the sequence \(\overline{P_k}(x,\cdot) := \frac{1}{k}\sum_{n=1}^{k} P^n(x,\cdot)\) is tight.

\begin{cor}
The Markov Chain \(\Phi\) has a compact state space, and so is bounded in probability on average.
\end{cor}

\subsection{Positive Harris Recurrent}

We can talk about the number of times a Markov chain lies in a particular set of states as \(\eta(\{x_i\}_{i=0}^{\infty},S) = |\{i|x_i \in S\}|\). This allows us to define Harris recurrent. A Markov chain \(\Phi\) is Harris recurrent if every set \(S\) with positive Borel measure has \(P(\eta(\Phi, S) = \infty) = 1\). A Markov chain is positive if there is a probability measure on the state space \(X\) which is invariant under iteration by the Markov process. Both of these properties are extremely useful, and a Markov chain which satisfies both is called positive Harris recurrent.

A state \(x^*\) is reachable if \(\sum_{n=0}^{\infty}P^n(x,O) > 0\) for every open set \(O\) containing \(x*\). We note that the regular polygon is reachable using Theorem~\ref{thm:FullConnected}. We again can utilize the book by Meyn and Tweedie by quoting proposition 18.3.2\cite{MeynTweedie93}.

\begin{theorem}
Suppose that \(\Phi\) is a \(T\)-chain with a reachable state. Then \(\Phi\) is positive Harris if and only if it is bounded in probability on average.
\end{theorem} 

This tells us that our Markov chain \(\Phi\) is positive Harris recurrent.

\subsection{Ergodic}

We seek to finish by using the main theorem from Meyn and Tweedies book\cite{MeynTweedie93}.

\begin{theorem}[Aperiodic Ergodic Theorem]
Suppose that \(\Phi\) is an aperiodic Harris recurrent chain, with invariant measure \(\pi\). Then \(\Phi\) being positive Harris recurrent is equivalent to the following.

For every initial condition \(x\in X\),
\begin{align}
sup_{S\in B(X)} |P^n(x,S)-\pi(S)|\rightarrow 0
\end{align}
as \(n\rightarrow \infty\), and moreover for any two regular initial distributions \(\lambda, \mu\),
\begin{align}
\sum_{n=1}^{\infty} \int \int \lambda(dx)\mu(dy) sup_{S\in B(X)} |P^n(x,S)-P^n(y,S)| < \infty.
\end{align}
\end{theorem}

We have already shown that \(\Phi\) is a positive Harris recurrent T-chain, so the aperiodic ergodic theorem shows that if the chain is aperiodic then it is ergodic. A chain is aperiodic if there is no period greater than one. A chain is periodic with period \(d\) if there is a collection of disjoint closed sets \(\{C_i\}_{i\in\Z_{d}}\) with the probability of going from \(C_i\) to \(C_{i+1}\) in exactly one step is one. 

\begin{theorem}
The Markov chain \(\Phi\) is aperiodic and therefore ergodic
\end{theorem}
\begin{proof}
Note that if the pair of points through which a potential reflection will take place is chosen to be distance exactly two from each other, then there is a choice of plane which contains the length two arc connecting them. This means that there is a choice of noise parameter which leaves the equilateral knot fixed. Thus, for any set \(C\), there is a move which takes a state in \(C\) to itself. Suppose \(C_i\) and \(C_{i+1}\) are disjoint non-empty closed sets. Then for each \(x\) in \(C_i\) there is a move \(m\) which fixes \(x\) and so takes \(x\) to \(m(x)\in X\setminus C_{i+1}\). This complement is an open set and \(m(x)\) is a continuous function of \(m\) so there is a positive probability of going from \(C_i\) to \(X\setminus C_{i+1}\). This ensures that no collection of multiple disjoint sets can satisfy the periodicity condition. This shows aperiodicity which then allows us to apply the aperiodic ergodicity theorem to conclude that the Markov chain is ergodic.
\end{proof}

\section{Conclusion and Future Work}

The reflection algorithm with the use of intermediate checks that we have presented here provides a greatly expanded means of analyzing knots with a thickness. This algorithm is significantly faster for sufficiently large thicknesses. Further, the proof that this method is ergodic for every positive thickness ensures that it can be used to analyze every possible feature of geometric knots. This new algorithm provides the means for the careful study of the effects of excluded volume across the entire range of lengths and thicknesses. I hope to have this algorithm used to study the impact of thickness, and to analyze the resulting probability distribution in the space of thick knots.

\bibliographystyle{amsplain}
\bibliography{dissertation.bib}

\begin{thebibliography}{10}

\bibitem{ACM11}
Sotero Alvarado, Jorge~Alberto Calvo, and Kenneth~C. Millett.
\newblock The generation of random equilateral polygons.
\newblock {\em Journal of Statistical Physicis}, 143:102--138, 2011.

\bibitem{CS13}
Jason Cantarella and Clayton Shonkwiler.
\newblock The symplectic geometry of closed equilateral random walks in
  3-space.

\bibitem{Champoux01}
James~J. Champoux.
\newblock Dna topoisomerases: structure, function and mechanism.
\newblock {\em Annual Review of Biochemistry}, 70:369--413, 2001.

\bibitem{Chern70}
S.S Chern.
\newblock {\em Studies in Global Geometry and Analysis}.
\newblock Prentice-Hall, 1970.

\bibitem{FiersSinsheimer62}
W.~Fiers and R.~L. Sinsheimer.
\newblock The structure of the dna of bacteriophage phi x174. iii.
  ultracentrifugal evidence for a ring structure.
\newblock {\em Journa of Molecular Biology}, 5:424--434, 1962.

\bibitem{Grunbaum2001333}
Branko Gr{\"u}nbaum and Joseph Zaks.
\newblock Convexification of polygons by flips and by flipturns.
\newblock {\em Discrete Mathematics}, 241(1–3):333 -- 342, 2001.

\bibitem{LDW76}
L.~F. Liu, R.~E. Depew, and J.~C. Wang.
\newblock Knotted single strand dna formed by treatment with escherichia coli w
  protein.
\newblock {\em Journal of Molecular Biology}, 106(2):439--452, 1976.

\bibitem{McLeish08}
T.~McLeish.
\newblock A tangled tale of topological fluids.
\newblock {\em Physics Today}, 61(8):40--45, 2008.

\bibitem{MeynTweedie93}
S.P. Meyn and R.L. Tweedie.
\newblock {\em Markov Chains and Stochastic Stability}.
\newblock Springer-Verlag, 1993.

\bibitem{Millett94}
Ken Millett.
\newblock Knotting of regular polygons in 3-space.
\newblock {\em Journal of Knot Theory and its Ramifications}, 3(3):263--278,
  1994.

\bibitem{MPR2007}
Kenneth~C. Millett, Michael Piatek, and Eric Rawdon.
\newblock Polygonal knot space near ropelength-minimized knots.
\newblock {\em Journal of Knot Theory and Its Ramifications}, 17(05):601--631,
  2008.

\bibitem{Randell88b}
Richard Randell.
\newblock Conformation spaces of molecular rings.
\newblock {\em Physical and Theoretical Chemistry}, 54:125--140, 1988.

\bibitem{Rawdon2000}
Eric~J. Rawdon.
\newblock Approximating smooth thickness.
\newblock {\em Journal of Knot Theory and Its Ramifications}, 09(01):113--145,
  2000.

\bibitem{Toussaint2005219}
Godfried Toussaint.
\newblock The {Erd\H{o}s \(-\) Nagy} theorem and its ramifications.
\newblock {\em Computational Geometry}, 31(3):219 -- 236, 2005.

\bibitem{WDC85}
S.A. Waserman, J.M. Dungan, and N.R. Cozzerelli.
\newblock Discovery of a predicted dna knot substantiates a model for
  site-specific recombination.
\newblock {\em Science}, 229(4709):171, 1985.

\end{thebibliography}

\end{document}